\documentclass{article}

\usepackage{amsthm}
\usepackage{amssymb}
\usepackage{amsmath}
\usepackage{graphicx}
\usepackage{eufrak} % Frakturgraphie
\usepackage{comment}
\usepackage{mathrsfs}
\usepackage{eucal}  % Kalligraphie
\usepackage{mathrsfs}
\usepackage{mathtools}
\usepackage{ifthen}
\newboolean{skript}
\newcommand{\beweis}{\ifthenelse{\boolean{skript}} {
\begin{proof}[Beweis]
\mbox{}\vfill \mbox{}
\end{proof}
\pagebreak }{} }

\allowdisplaybreaks[3]

\providecommand{\Fourier}{\mathcal{F}}

\DeclareMathOperator{\rem}{\text{R}}

\usepackage{multirow}
\usepackage{multicol}

\let\scr\mathscr     % More beautiful calligraphic

\renewcommand{\hat}{\widehat}
\renewcommand{\tilde}{\widetilde}

\newcommand{\normz}[1]{\left\| #1 \right\|_{\lk^2}   }

\providecommand{\intu}{\int_0^{|u|}\limits   }

\usepackage{titlesec}
\usepackage{textfit}

\usepackage{geometry}
\newcommand{\ebig}[1]{\mathbb{E}\big[  #1  \big]  }
\newcommand{\Ebig}[1]{\mathbb{E}\Big[  #1  \Big]  }
\newcommand{\Ebigg}[1]{\mathbb{E}\Bigg[  #1  \Bigg]  }
\newcommand{\ebigg}[1]{\mathbb{E}\bigg[  #1  \bigg]  }

\newcommand{\ecfLV}{\hat{\phi}_X^{ {\scriptscriptstyle L \hspace*{-0,05cm}V }  }            }
\newcommand{\ecfmod}{\hat{\phi}_X^{ {\scriptscriptstyle m \hspace*{-0,05cm} o\hspace*{-0,05cm} d }  }            }

\providecommand{\m}{\text{-}}

\usepackage{enumitem}

\input{mymath_art.sty}

\renewcommand{\d}{\,  \textrm{d}  }
\usepackage[round,authoryear]{natbib}

\mathtoolsset{showonlyrefs}

\newcommand{\dpsi}{\frac{\partial}{\partial u_1}  \psi }
\newcommand{\dpsiemp}{\frac{\partial}{\partial u_1}  \hat{\psi} }
\newcommand{\psitilde}{ \tilde{\psi} }
\newcommand{\psiemp}{\hat{\psi}}

\newcommand{\inth}{ \int_{-1/h}^{1/h} \limits }

\newenvironment{keywords}%
   {\begin{trivlist}\item[]{\bfseries\sffamily Keywords:}\ }
   {\end{trivlist}}
\newenvironment{AMS}%
   {\begin{trivlist}\item[]{\bfseries\sffamily AMS Subject Classification:}\ }
   {\end{trivlist}}
\begin{document}
\title{Density deconvolution from repeated measurements without symmetry assumption on the errors}
\author{Fabienne Comte\thanks{MAP 5, UMR CNRS 8145, Universit\'e Paris D\'escartes, France }   \   and 
Johanna Kappus  \thanks{Corresponding author; Institut f\"ur Mathematik, Universit\"at Rostock, Germany,  e-mail: johanna.kappus@uni-rostock.de }  }
\maketitle 
\begin{abstract}
\noindent We consider deconvolution from repeated observations with unknown error distribution. So far, this model has mostly been studied under the additional assumption that the errors are symmetric. 

\noindent We construct  an estimator for the non-symmetric error case and study its theoretical properties and practical performance.  It is interesting to note that we can improve substantially upon the rates of convergence which have so far been presented in the literature and, at the same time, dispose of most of the  extremely restrictive assumptions which have been imposed so far. 
\end{abstract}
\section{Introduction}\setlength{\parindent}{0pt}
Density deconvolution is  one of the classical topics in nonparametric statistics and  has been extensively studied during the past decades. The aim is to identify the density of some random variable $X$, which cannot be observed directly, but is contaminated by some additional additive error $\eps$, independent of $X$. 

A large amount of literature is available on the case where the distribution of the errors is perfectly known. To mention only a few of the various publications on this subject, we refer to \cite{Caroll_Hall},  \cite{Stefanski}, \cite{Stefanski_Carroll}, \cite{Fan}, \cite{Efromovich}, \cite{Pensky}, \cite{Comte_Rosenholz}.  

However, perfect knowledge of the error distribution is hardly ever realistic in applications. For this reason, the interest in deconvolution problems with unknown error distribution has grown. \cite{Meister_miss} has investigated deconvolution with misspecified error distribution.  \cite{Diggle_Hall}  replace the unknown characteristic function of the errors by its empirical counterpart and then apply standard kernel deconvolution techniques.  The effect of estimating the characteristic function of the errors is then  systematically studied by \cite{Neumann}. Let us also mention  \cite{Johannes} for deconvolution problems with unknown errors.

The last mentioned publications have been working under the standing  assumption that an additional sample of the pure noise is available. This is realistic in some practical examples. For example, if  the noise is due to some measurement error, it is possible to carry out additional measurements in absence of any signal.

 However, in many fields of applications it is not  realistic to assume that an additional training set is available. It is clear that, to make the problem identifiable, some additional information on the noise is required.  In the present work, we are interested in the case where  information can be drawn from repeated measurements of $X$, perturbed by independent errors.   This framework is known as model of repeated measurements or panel data model.  The observations are of the type 
\begin{equation}
Y_{j,k}= X_j+\eps_{j,k};\  \  j= 1, \cdots, n; \   \   k=1, \cdots, N, 
\end{equation}
where all $X_j$ and $\eps_{j,k}$ are independent. 

This problem is relatively well-studied under the assumption that the distribution of the errors is symmetric. We refer to \cite{Delaigle_repeated}, \cite{Comte_repeated} and \cite{Kappus_Mabon}. 

In the present paper, we consider deconvolution from repeated observations when the symmetry assumption on the errors is no longer satisfied.   The estimation strategies which have been developed for the symmetric error case cannot be generalized to this framework and a completely different approach is in order.  The same problem has been investigated in earlier publications    by \cite{Li_Vuong} and by \cite{Neumann_2}. 

The paper by Li and Vuong has two major drawbacks. On one hand, the rates of convergence presented therein are extremely slow, in comparison to the rate results which are usually found in deconvolution problems. On the other hand, the mentioned authors impose extremely restrictive assumptions on the target density and on the distribution of the noise, which are only met in some exceptional cases. 

Neumann succeeds in overcoming this second drawback and constructing consistent estimators under most general assumptions. However, rate results are not given in that paper so the question whether the convergence rates found by Li and Vuong can be improved has so far remained unanswered. 
Moreover, the estimator proposed by Neumann is only implicitly given and non-constructive, so it is difficult to investigate the practical performance. 

In the present work, we study  a fully constructive estimator, which is based on a modification of the original procedure by Li and Vuong.  It is interesting to note that we are able to improve substantially upon the rates of convergence found by Li and Voung and, at the same time, dispose of most of their restrictive assumptions. Surprisingly, it can also be shown that our estimator outperforms, in some cases, the estimators which have been studied for the structurally simpler case of repeated observations with symmetric errors. 

This paper is organized as follows: In Section 2, we introduce the statistical model and define estimators for the target density, as well as for the residuals. In Section 3, we provide upper risk bounds and derive rates of convergence. In Section 4, we present some data examples.  All proofs are postponed to Section 5. 
\section{Statistical model an estimation procedure}
Let $\eps_1$ and $\eps_2$ be independent copies of a random variable $\eps$ and let $X$ be independent of $\eps_1$ and $\eps_2$.  By $Y$, we denote  the random vector $Y=(Y_1, Y_2)=(X+\eps_1, X+\eps_2)$. We observe $n$ independent copies 
\[
Y_j=(Y_{j,1}, Y_{j,2}), \  j=1, \cdots, n
\]
of $Y$.   The following assumptions are imposed  on $X$ and $\eps$:
\begin{itemize}
\item[(A1)]  $X$ and $\eps$  have a square integrable Lebesgue densities  $f_X$ and $f_{\eps}$. 
\item[(A2)] The characteristic functions $\phi_{\eps}(\cdot) = \E[e^{i \cdot \eps }  ]$ and $\phi_X(\cdot)=\E[e^{i\cdot X}]$ vanish nowhere.
\item[(A3)]  $\E[\eps]=0$. 
\end{itemize}
Our objective is to estimate $f_X$ and $f_\eps$.  This statistical framework allows a straightforward generalization to the case where more than two observations of the noisy random variable $X$ are feasible. However, for sake of simplicity and clarity, we content ourselves with considering the two dimensional case. 

In the sequel, we denote by $\psi$ the characteristic function of the two dimensional random vector $Y$, 
\begin{equation}
\psi(u_1, u_2)  = \E[ e^{i (u_1 Y_1 + u_2 Y_2) }  ].
\end{equation}
By independence of $X$, $\eps_1$ an $\eps_2$, the following  holds for $\psi$: 
\begin{equation}
\label{Strukturformel}
\psi(u_1, u_2)= \E[e^{i(u_1+u_2)X  } e^{iu_1\eps_1} e^{iu_2 \eps_2}  ]
=\phi_X(u_1+u_2)\phi_{\eps}(u_1) \phi_{\eps}(u_2). 
\end{equation}
From formula \eqref{Strukturformel} one derives  the following Lemma, which has  been formulated and proved in \cite{Li_Vuong}. Lemma 2.1 is then the key to the construction of the estimator. 
\begin{lemma} Assume that $\E[|Y_1|]<\infty$ and $\E[\eps]=0$.  Then $\phi_X$ is determined  by $\psi$ via the following formula:
\begin{equation}
\label{cf_X}
\phi_X(u) = \exp \int_0^u   \frac{\dpsi(0,u_2)   }{ \psi(0,u_2)   }  \d u_2.
\end{equation}
\end{lemma}
Li and Vuong propose the following estimator of $\phi_X$:
\begin{equation}
\label{Schaetzer_cf_LV}
\hat{\phi}_X^{ {\scriptscriptstyle L \hspace*{-0,05cm}V }  }(u):= \exp \int_0^u  \frac{\dpsiemp(0,u_2)}{\psiemp(0,u_2)}  \d u_2, 
\end{equation}
with 
\begin{equation}
\psiemp(u_1,u_2)= \frac{1}{n}  \sum_{j=1}^n  e^{i(u_1 Y_{j,1}+ u_2 Y_{j,2} )  } \   \   \ 
\text{and}  \   \    \
 \dpsiemp(u_1,u_2)= \frac{1}{n} \sum_{j=1}^n  i Y_{j,1}  e^{i (u_1 Y_{j,1}+ u_2 Y_{j,2} ) }
\end{equation}
denoting the empirical version of $\psi$ and its first partial derivative. 

Given a kernel $\kf$ and bandwidth $h$, the corresponding estimator of $f_X$ is 
\begin{equation}
\label{Schaetzer_Li}
\hat{f}_{{X_{h}} }^{{\scriptscriptstyle L\hspace*{-0,05cm} V  }} (x) = \frac{1}{2\pi}  \int e^{-iux } \ecfLV(u) \Fourier \kf_h(u)  \d u, 
\end{equation}
with $\kf_h(\cdot) = 1/h \kf(\cdot/h)$ and with 
$\Fourier \kf_h(u)=\int e^{iux}\kf_h(x) \d x$  denoting the Fourier transformation. 
We propose a  modified version of $\hat{f}_{{  X_{{h}}  }}^{ {\scriptscriptstyle L\hspace*{-0,05cm}V} }$. First of all, it is well  known that  small values of the denominator lead to unfavorable effects in the estimation procedure, so   it is preferable to consider some regularized version of $\hat{\psi}$. One  possible approach is to 
replace $\hat{\psi}$ in the denominator by  $\psiemp+\rho$ with some Ridge-parameter to be appropriately chosen. See, for example, \cite{Delaigle_repeated}.  However, following ideas in  \cite{Neumann}, we prefer to define
\begin{align}
\psitilde(0,u_2) = \frac{ \hat{\psi}(0,u_2) }{\min\{ n^{1/2} |\hat{\psi}(0,u_2)| , 1 \} }. 
\end{align}
and use $1/ \psitilde(0,u_2)$ as an estimator of $1/\psi(0,u_2)$. 

This leads  to defining the following modified estimator of $\phi_X$: 
\begin{equation}
\label{erster_Schaetzer}
\ecfmod(u) = \exp \int_0^u  \frac{\dpsiemp(0,u_2)}{\psitilde(0,u_2)}  \d u_2.
\end{equation}
Next, we have to pay attention to the fact that, by definition, neither $\ecfLV$ nor $\ecfmod$ need to be characteristic functions and they  may take values larger than one. Indeed, much of the complexity in the proofs presented in \cite{Li_Vuong} and many of the restrictive  assumptions  imposed therein are a consequence of the fact that there appears an unbounded exponential term in the definition of $\ecfLV$ which has to be controlled, leading to some Bernstein-type arguments and hence to the assumption that the supports are bounded. 

However, the quantity to be estimated  is, in any case, a characteristic function, so the quality of the estimator can be improved by bounding the absolute value of  $\ecfmod$. These considerations lead to defining our final estimator of the characteristic function of $X$, 
\begin{equation}
\label{Schaetzer_X_final}
\hat{\phi}_X(u) : =  \frac{ \ecfmod(u)  }{ \max\{ 1  ,\abs{ \ecfmod(u)    } \}      }.
\end{equation}

Sometimes, one may not only be interested in the estimation of the target density itself, but also in the distribution of the residuals.  The following holds true for the characteristic function of $\eps$: 
\begin{equation}
\phi_\eps(u) = \frac{\psi(0,u)  }{\phi_X(u)}. 
\end{equation}
This quantity can hence  be recovered, using a plug-in estimator.  We set
\[
\tilde{\phi}_X(u) = \frac{\hat{\phi}_X(u)}{\min\{ n^{1/2} |\hat{\phi}_X(u)|, 1 \}  }
\]
and then   
\begin{equation}
\label{Schaetzer_eps_final}
\hat{\phi}_{\eps}(u): = \frac{\hat{\psi}(0,u) }{\tilde{\phi}_X(u)  }.
\end{equation}

Given a kernel  $\kf$ and bandwidth $h>0$, the  kernel estimators of $f_X$ and $f_{\eps}$ corresponding to formula \eqref{Schaetzer_X_final} and 
\eqref{Schaetzer_eps_final} are 
\begin{equation}
\hat{f}_{{ X_{h} }} (x)  = \frac{1}{2\pi} \int e^{- iux }  \hat{\phi}_X(u) \Fourier \kf_h(u) \d u 
\end{equation}
and 
\begin{equation}
\hat{f}_{{\eps_{h} }} (x)  = \frac{1}{2\pi} \int e^{- iux }  \hat{\phi}_{\eps}(u) \Fourier \kf_h(u) \d u .
\end{equation}
\section{Risk bounds and rates of convergence}
\subsection{Non-asymptotic risk bounds}
We start by analyzing the performance  of $\hat{f}_X$.  It is important to stress that we can dispose of most of the assumptions which have been imposed in earlier publications on the subject. Indeed the conditions on $X$ and $\eps$ which are imposed in \cite{Li_Vuong}, namely boundedness of the support of $f_X$ and $f_\eps$ and nowhere  vanishing characteristic functions, are violated for any distribution which is commonly studied in probability theory. In  \cite{Bonhomme}  an estimator is constructed under weaker assumptions on the distributions. But still, it is required in that paper that
$X$ have moments of all orders, which is certainly quite restrictive. Moreover, the rates which are found by those authors turn out to  be  even slower than the rate results presented in \cite{Li_Vuong}.  It is interesting to note that we can substantially improve upon these results, even though our assumptions are much weaker. 

In \cite{Neumann_2}, an implicit estimator of $f_X$  is proposed.   The strength of this approach lies in the fact that it is fully general. However, the price one has to pay is the lack of constructivity. The estimator is found as the solution to an abstract minimization problem, so the practical computation is not clear. Moreover, consistency of the estimator is shown, but rate results are not given, so nothing can be said about the quality of the  procedure. 

Finally, \cite{Delaigle_repeated} and \cite{Comte_repeated} have studied estimators in a repeated measurement model, but it is assumed in both papers that the distribution of the noise is symmetric. It is the main concern of the present publication to be able to dispose of the symmetry. 

In the sequel, we impose the following mild regularity assumption on the characteristic function of $X$: 
\begin{itemize}
\item [(A4)]  For some positive constant $C_{X}$, 
\[
\forall u, v \in \R_+:\    \    ( v\leq u ) \Rightarrow 
( |\phi_X(u) |    \leq  C_{X} |\phi_X(v)|  ). 
\]
\end{itemize}
The following bound can be given on the mean integrated squared error:
\begin{theorem}
\label{Schranke_X}  Let $\kf$ be supported on $[-1,1]$.  Assume that (A1)-(A4) are satisfied and that  for some positive integer $p$,  $\E[|Y_1|^{2p}]<\infty$.
Assume, moreover, that $\phi_X''\phi_{\eps}$ is integrable. Then for some positive constant $C$ depending only on $p$, 
\begin{align}
& \Ebig{    \normz{f_X- \hat{f}_{{X_h}}  }^2  } 
 \leq  2 \normz {  f_X - \kf_h\ast f_X }^2  \\
& + \frac{C C_X  G(X,\eps,1,1/h) }{n}  \inth  \int_0^{|u|}\limits  \frac{1}{|\phi_\eps(z)|^2 }  \d z \d u 
+ C   G(X,\eps, p, 1/h)  \inth \bigg(\frac{1}{n}  \int_0^{|u|}\limits  \frac{1}{|\phi_Y(z)|^2 }  
\d z    \bigg)^p    \d u
,
\end{align}
with 
\[
\phi_Y(z)=\E[e^{iz Y_2} ] =\psi(0,z) = \phi_X(z)\phi_{\eps}(z)
\]
and 
\begin{align}
G(X,\eps,p,u)  := &( \|\phi''_X\phi_{\eps} \|_{\lk^1}  + \E[\eps^2 ] \| \phi_X\phi_{\eps} \|_{\lk^1}   +\|\phi_X' \phi_\eps\|_{\lk^2}^2 )^p 
+\Big( \intu  {|\frac{\partial}{\partial u_1} \log \psi(0,x) |^2 }  \d x  \Big)^p  \\
& +   \frac{ \E[|Y_1|^{2p} ] 1_{\{p\geq 2\}  } u^p }{ n^{p-1}  }
+  \E^{\frac{1}{2} }[|Y_1|^{2p} ] .  
\end{align}
\end{theorem}

In analogy with (A4), we impose the following assumption on the characteristic function of the errors: 
\begin{enumerate}
\item[(A5)] For some positive constant  $C_{\eps}$, 
\[
\forall u,v \in \R_+: \  ( v\leq u ) \Rightarrow ( |\phi_{\eps}(u)|\leq C_{\eps}|
\phi_{\eps}(v)|).
\]
\end{enumerate}
The following bound can then be given on the mean integrated squared error of $\hat{f}_{\eps}$: 
\begin{theorem}\label{Schranke_eps} Assume that $\kf$ is supported on $[-1,1]$  and the assumptions (A1)-(A5) are met. Assume, moreover, 
that for some positive integer $q\geq 2$,  $\E[|Y_1|^{4q}]$ is finite. Then for some positive constant  $C$ depending only on $q$, 
\begin{align}
& \Ebig{\normz{f_{\eps} - \hat{f}_{{\eps_h }}   }^2 }
\leq \normz{ f_{\eps}  - \kf_h\ast f_{\eps}  }^2   +CC_{\eps} \Big[    \frac{ G(X,\eps,1,1/h)}{n }\inth \intu \frac{1}{|\phi_X(z)|^2 }  \d z  \d u  \\
&+\frac{G(X,\eps,q,1/h)}{n^{q-1}  }   \inth \frac{1}{|\phi_X(u)|^2 } 
\bigg( \intu \frac{1}{|\phi_X(z)|^2}   \d z  \bigg) \bigg( \intu \frac{1}{|\phi_Y(z)|^2 }  \d z   
 \bigg)^{q-1} \d u  \\
&  +\frac{G(X,\eps,2,1/h)^{1/2}  }{n^2}  \inth \frac{1}{|\phi_X(u)|^2 }  \bigg( 
\intu    \frac{1}{|\phi_Y(z)|^2}  \d z  \bigg)   \d u\\  
& + \frac{G(X,\eps,2q,1/h) ^{1/2} }{n^q}  \inth \frac{1}{|\phi_X(u)|^4 }  
 \bigg( \intu \frac{1}{|\phi_Y(z)|^2}  \d z   \bigg)^q\d u +  \frac{1}{n^2} \inth \frac{1}{|\phi_X(u)|^4}  \d u \Big].
\end{align} 
\end{theorem}
\   \\
\textbf{Discussion}
It is easily seen that the assumptions (A4) and (A5) are not very restrictive.  They are met, for example, for normal or mixed normal distributions, Gamma distributions, bilateral Gamma distributions and many others. By a location shift, one can always ensure that $\E[\eps]=0$.

The integrability condition on $\phi_X''\phi_\eps$ is also very mild. Under the above assumptions, it is automatically met if $\phi_\eps$ is integrable but can also be checked in most other cases. 

The upper bound in Theorem 3.1 differs from the bounds which are commonly found in deconvolution problems in two ways: For one thing, there appears an additional inner integral in the variance term. This could be a consequence of the two-dimensional nature of the underlying problem.  On the other hand, it is completely unexpected to find, in the second variance term, the characteristic function of the target density appearing in the denominator.  On an intuitive level, this phenomenon could be understood as follows:  To draw inference on $X$ some information on the noise is required. However, in comparison to standard deconvolution problems, $\eps$ is itself an unobservable quantity and is contaminated by $X$. Consequently, $X$ does not only play the role of a random variable of interest but also, with respect to the error term, the role of a contamination.
This might explain the occurrence of $\phi_X$ in the denominator.  

\subsection{Rates of convergence}
In what follows, we derive rates of convergence under regularity assumptions on the target density  $f_X$ and on the density  $f_{\eps}$ of the noise.  For sake of simplicity, we assume in this section that $\kf$ is the sinc-kernel, $\Fourier \kf=1_{[-1,1]}$. 

Let us introduce some notation: For $\rho, C_1>0$, $\beta, c\geq 0$, $C_2\geq 1$, we  denote by $\mathcal{F}^u(C_1,C_2,c,\beta, \rho)$ the class of  square integrable densities $f$ such that the  characteristic function $\phi(\cdot)= \int e^{i\cdot x} f(x) \d x$ satisfies 
\begin{equation} \label{Monotonie}
\forall u,v \in \R^+:\    \       (u\geq v) \Rightarrow ( |\phi(u)|\leq C_2 |\phi(v)| )
\end{equation}
and
\begin{equation} 
 \forall u\in \R: \   \   |\phi(u)|  \leq  (1+C_1 |u|^2)^{-\frac{\beta}{2} } e^{ -c|u|^{\rho}  } .
\end{equation}
If $c=0$, the functions collected in $\mathcal{F}^u(C_1,C_2,c,\beta, \rho)$ are called 
\textit{ordinary smooth}. For $c>0$, they are called \textit{supersmooth}. 
By  $\mathcal{F}^\ell(C_1,C_2,c,\beta, \rho)$, we denote the class of square integrable densities for which \eqref{Monotonie} holds and, in addition, 
\begin{equation} 
\forall u\in \R: \    \    |\phi(u)| \geq   (1+C_1 |u|^2)^{-\frac{\beta}{2} } e^{ -c|u|^{\rho}  } .
\end{equation}
For $C_3>0$, we denote by $\mathcal{G}(C_3,p)$ the class of pairs $(f_{X} , f_{\eps}  )$ of square integrable densities for which the following conditions are met:  For the characteristic function $\phi_X$ of $f_X$ and $\phi_\eps$ of $f_\eps$, 
\begin{align}
( \| \phi_X'' \phi_\eps    +  \E[\eps^2 ] \phi_X \phi_\eps  \|_{\lk^1 }   +  \|\phi_X'\phi_\eps  \|_{\lk^2 }^2 )^p
+\E[|X+\eps|^{2p}  ]  \leq C_3
\end{align}
holds and moreover,  $(\log  \phi_{X+\eps} )'$ is square integrable, with 
\begin{align}
 \|  (\log  \phi_{X+\eps} )'  \|_{\lk^2}^{2p}  \leq C_3.
\end{align}
Finally, we use the short notation 
\begin{align}
\mathcal{F}^{u,\ell}(X,\eps,p) =\Big[ \mathcal{F}^u(C_{1,X}, C_{2,X},c_X,\beta_X, \rho_X) \times \mathcal{F}^{\ell} (C_{1,\eps}, C_{2,\eps},c_\eps,\beta_\eps, \rho_\eps) \Big]  \cap \mathcal{G}(C_3,p)
\end{align}
and 
\begin{align}
\mathcal{F}^{\ell,u}(X,\eps,p) =\Big[ \mathcal{F}^{\ell}(C_{1,X}, C_{2,X},c_X,\beta_X, \rho_X) \times \mathcal{F}^{u } (C_{1,\eps}, C_{2,\eps},c_\eps,\beta_\eps, \rho_\eps) \Big]  \cap \mathcal{G}(C_3, p ).
\end{align}
\subsubsection*{Estimation of the target density}
We start by providing rates of convergence for the estimation of $f_X$.  Let $p\geq 2$. We may limit the considerations to bandwidths $h\geq n^{-1/2}$, so the term $\frac{u^p}{n^{p-1} }$ appearing in the definition of $G(X, \eps, p,u)$ is readily negligible.  We consider three different cases:\\[0,2cm]
\textbf{Case I:} Ordinary smooth density with ordinary smooth errors , $c_X = c_{\eps}=0, \beta_X >1/2, \beta_{\eps}>1/2$. 

Then  the choice of the kernel, Theorem \ref{Schranke_X} and the definition of $\mathcal{F}^{u,\ell}(X,\eps,p) $ give
\begin{align}
\sup_{(f_X,f_\eps)\in \mathcal{F}^{u,\ell}(X,\eps,p) }  \Ebig{  \|f_X -\hat{f}_{X,h}\|_{\lk^2}^2   } = O (r_{n,h} )  
:= O \Big( (1/ h)^{\gamma_1} +\frac{1}{n} (1/h)^{\gamma_2}    + \frac{1}{n^p} (1/h)^{\gamma_3}   \Big)
\end{align}
with 
\begin{align}
\gamma_1= \m 2\beta_X +1, \  \ \gamma_2= 2\beta _{\eps} +2,\   \
\gamma_3 = p(2\beta_X+2\beta \eps +1)+1.
\end{align}
Minimizing $r_{n,h}$ with respect to $h$ yields for the optimal bandwidth $h^*$, 
\[
1/h^* \asymp  n^{ \frac{ 1}{2 \beta_{\eps}  +2(1+1/p) \beta_X + 1}  }.
\]
Plugging $h^*$ in gives 
\begin{align}
 \sup_{(f_X,f_\eps)\in \mathcal{F}^{u,\ell} (X,\eps,p) } \Ebig{  \|f_X -\hat{f}_{X,h^*}\|_{\lk^2}^2   }
= O \Big(  n^{-\frac{  (2\beta_X - 1) }{2 \beta_{\eps}  +2(1+1/p) \beta_X + 1}  }   \Big). 
\end{align} 
\textbf{Case II} Ordinary smooth density with supersmooth errors, $\beta_X>1/2, c_X=0, c_{\eps} >0$.
Then
\begin{align}
\sup_{(f_X,f_\eps)\in \mathcal{F}^{u,\ell} (X,\eps,p) }&  \Ebig{  \|f_X -\hat{f}_{X,h}\|_{\lk^2}^2   }  =O(r_{n,h} )  \\
&:=  O \Big( (1/h)^{\gamma_1}  + \frac{1}{n}  (1/h)^{\gamma_2 }  
\exp( 2 c_{\eps} (1/h)^{\rho_{\eps}  }  )   +\frac{1}{n^p}  (1/h)^{\gamma_3 }\exp(2 p  c_{\eps} (1/h)^{\rho_{\eps}  }  ) \Big),
\end{align}
with 
\begin{align}
\gamma_1=-2\beta_X+1 ,\    \
\gamma_2= [ (2\beta_{\eps} +1-\rho_{\eps} )_+ +1-\rho_{\eps}  ]_+, \   \  
\gamma_3 = [ p (2\beta_{\eps}+2\beta_X +1-\rho_{\eps}  )_+ +1-\rho_\eps ]_+.
\end{align}
Selecting $h^*$ as the minimizer of $r_{n,h}$ gives 
\begin{align}
1/h^* = \left(  \frac{1}{2 c_{\eps} }  (\log n) -\frac{1}{2 c_{\eps}  }  \log (\log n)^{\gamma }  +O(1)  \right)^{1/\rho_{\eps}  }.
\end{align}
with 
\begin{align}
\gamma = \max\{ \frac{ \gamma_2}{\rho_\eps }  +2\beta_X -1 , 1/p( \frac{ \gamma_3}{\rho_\eps}  +2\beta_{X} - 1)  \}.   
\end{align}
From this we derive that 
\begin{align}
\sup_{(f_X,f_\eps)\in \mathcal{F}^{u,\ell} (X,\eps,p) }  \Ebig{  \|f_X -\hat{f}_{X,h^*}\|_{\lk^2}^2   }   = O \left( (\log n)^{- \frac{2 \beta_X  -1}{\rho_{\eps}    }  }  \right). 
\end{align}
\textbf{Case III} Supersmooth density with ordinary smooth errors, $c_X >0, c_{\eps} = 0, \beta_\eps >1/2$.  In this case, 
\begin{align}
\sup_{(f_X,f_\eps)\in \mathcal{F}^{u,\ell} (X,\eps,p) }&  \Ebig{  \|f_X -\hat{f}_{X,h^*}\|_{\lk^2}^2   }   = O(r_{n,h} ) \\
&  := 
O \bigg(   (1/h)^{\gamma_1  }  \exp( - 2  c_X (1/h)^{\rho_X}  )    +\frac{1}{n} (1/h)^{\gamma_2}  +\frac{1}{n^p} 
(1/h)^{ \gamma_3}  \exp ( 2 p c_X (1/h)^{\rho_X} \bigg),
\end{align}
with 
\begin{align}
\gamma_1= - 2\beta_X +1- \rho_X ,\   \    \gamma_2= 2\beta_\eps+2,
\    \   \gamma_3= [ p (2\beta_X+2\beta_{\eps}+1-\rho_X )_+  + 1-\rho_X]_+ .
\end{align}
Minimizing $r_{n,h}$ yields 
\begin{align}
1/h^* = \left(  \frac{p}{2 c_{X} (p+1) }  (\log n) -\frac{1}{2 c_{X}  }  \log (\log n)^{\gamma}  +O(1)  \right)^{1/\rho_{X}  }
\end{align}
with 
\begin{align}
\gamma = \frac{\gamma_3 - \gamma_1}{(p+1)\rho_X}
 \end{align}
Then it holds that 
\begin{align}
\sup_{(f_X,f_\eps)\in \mathcal{F}^{u,\ell} (X,\eps,p) } &  \Ebig{  \|f_X -\hat{f}_{X,h^*}\|_{\lk^2}^2   }   =
 O\left( n^{-\frac{p}{p+1}  } (\log n)^{ \frac{  p/(p+1) \gamma_1+1/(p+1)\gamma_3
 }{ \rho_X  }    }   \right).
\end{align}

 \subsubsection*{Estimation of the residuals} In analogy with the rates for the estimation of $f_X$, we consider the following different cases:\\[0,2cm]
\textbf{Case I} Both, $f_{\eps}$ and $f_X$ are ordinary smooth, $c_X= c_{\eps}=0, \beta_X>1/2, \beta_{\eps}>1/2$. Then by Theorem \ref{Schranke_eps}  and by the definition of $\mathcal{F}^{\ell, u}(X,\eps,p)$, 
\begin{align}
\sup_{(f_X,f_\eps)\in \mathcal{F}^{\ell,u} (X,\eps, p)} & \Ebig{  \| f_{\eps} - \hat{f}_{\eps, h  }  \|_{\lk^2}^2  } =O ( r_{n,h}) \\
& :=  O\left((1/h)^{\gamma_1}  + \frac{1}{n} (1/h)^{\gamma_2} +\frac{1}{n^2}(1/h)^{\gamma_3} +\frac{1}{n^{p-1}  } (1/h)^{\gamma_4}  +\frac{ 1}{n^p} (1/h)^{\gamma_5} \right),
\end{align}
with 
\begin{align}
&\gamma_1 = - 2\beta_\eps +1,  \    \  \gamma_2= 2\beta_X+2,   \   \
\gamma_3 = 4\beta_X+2\beta_\eps +2, \   \  \gamma_4 =2(p+1)\beta_X + 2(p-1)\beta_{\eps} + p+1,\\
& \gamma_5=  2(p+2)\beta_X +2p\beta_{\eps} + p+1. 
\end{align}
Minimizing  $r_{n,h}$ gives 
\[
1/h^*\asymp  n^{ \frac{1}{ 2(1+2/(p-1) )\beta_X + 2(1+1/(p-1) ) \beta_{\eps} + 1+1/(p-1) }  }.
\]
Consequently,
\begin{align}
\sup_{(f_X,f_\eps)\in \mathcal{F}^{\ell,u} (X,\eps, p)}  \Ebig{  \| f_{\eps} - \hat{f}_{\eps, h^*  }  \|_{\lk^2}^2  } =O\left(n^{- \frac{2\beta_\eps - 1}{ 2(1+2/(p-1) )\beta_X + 2(1+1/(p-1) ) \beta_{\eps} + 1+1/(p-1) }  }   \right). 
\end{align}
\textbf{Case II}  Ordinary smooth  $f_{\eps}$   and supersmooth $f_X$,   $c_X>0, c_\eps=0,\beta_\eps>1/2$.  In this case, 
\begin{align}
 \sup_{(f_X,f_\eps)\in \mathcal{F}^{\ell,u} (X,\eps, p)}  &  \Ebig{  \| f_{\eps} - \hat{f}_{\eps, h  }  \|_{\lk^2}^2  } =O ( r_{n,h})  \\
 :=  &  O\bigg((1/h)^{\gamma_1} + \frac{1}{n} (1/h)^{\gamma_2} \exp(2 c_X (1/h)^{\rho_X} )    +\frac{1}{n^2} (1/ h) ^{\gamma_3   }\exp(4 c_X (1/h)^{\rho_X} )    \\
 &
 +\frac{1}{n^{p-1 } } (1/h)^{ \gamma_4}  \exp(  2c_X(p+1) (1/h)^{\rho_X}  ) 
+\frac{1  }{n^p}(1/h)^{\gamma_5} \exp( (2c_X(p+2) (1/h)^{\rho_X} ) \bigg) ,
\end{align}
with 
\begin{align}
& \gamma_1= \m 2\beta_\eps +1, \    \     \gamma_2 = [  (2\beta_{X}+1-\rho_X)_+  +1-\rho_X]_+, \\ &\gamma_3=[(2\beta_X+2\beta_\eps +1-\rho_X)_+   +2\beta_X +1-\rho_X]_+,  \\ 
& \gamma_4= [ (p-1)(2\beta_X +2\beta_\eps +1-\rho_X)_+ +(2\beta_X +1-\rho_X)_+ +2\beta_X+1-\rho_X]_+, \\
&     \gamma_5=   [  p(2\beta_X +2\beta_\eps +1-\rho_X)  +4\beta_X +1-\rho_X]_+.
\end{align}
Then 
\[
1/h^* = \left(  \frac{(p-1)}{ 2  c_X(p+1)}\log n     -  \frac{1}{2c_X} \log \left(  \log n\right)^{\gamma}  +O(1)
\right)^{1/\rho_X} 
\]
with  
\[
\gamma= 1/(p+1)  (\frac{\gamma_4}{\rho_X} +2\beta_\eps -1).
\]
 This implies 
\begin{align}
\sup_{(f_X,f_\eps)\in \mathcal{F}^{\ell,u} (X,\eps, p)}  \Ebig{  \| f_{\eps} - \hat{f}_{\eps, h^*  }  \|_{\lk^2}^2  } =O\left( (\log n)^{-\frac{ 2\beta_{\eps}  -1 }{\rho_X} }  \right). 
\end{align}
\textbf{Case III}  $f_{\eps}$ is supersmooth and $f_X$ is ordinary smooth. Then 
\begin{align}
&\sup_{(f_X,f_\eps)\in \mathcal{F}^{\ell,u} (X,\eps, p)}  \Ebig{  \| f_{\eps} - \hat{f}_{\eps, h  }  \|_{\lk^2}^2  } =O ( r_{n,h}) := O\bigg( (1/h)^{\gamma_1}  \exp( - 2 c_\eps (1/h)^{\rho_\eps}  )
+ \frac{1}{n } (1/ h) ^{\gamma_2}    \\
&+  \frac{ 1  }{n^2}(1/h)^{\gamma_3}  \exp(2 c_\eps (1/h)^{\rho_\eps} ) 
+\frac{1}{ n^{p-1}     } (1/h)^{\gamma_4}  \exp ( 2c_{\eps} (p-1) (1/h)^{\rho_{\eps} } )  
+\frac{(1/h)^{\gamma_5}}{n^p}   \exp ( 2c_{\eps} q (1/h)^{\rho_{\eps} } )\bigg) ,
\end{align}
with 
\begin{align}
&\gamma_1 = \m 2\beta_\eps +  1- \rho_\eps,  \    \   \gamma_2= 2\beta_X+2,  \    \   
\gamma_3 = [2\beta_X +2\beta_\eps +1-\rho_\eps)_+   +2\beta_X+1 -\rho_\eps ]_+ , \\
&  \gamma_4= [(p-1) (2\beta_X +2\beta_\eps +1-\rho_\eps)_+ +4\beta_X +2-\rho_\eps ]_+,\\  
& \gamma_5 = [p(2\beta_X+2\beta_\eps +1-\rho_\eps)_+ +4\beta_X+1-\rho_\eps ]_+.
\end{align}
We arrive at 
\begin{align}
1/h^*= \left(   \frac{(p-1)}{p 2c_\eps  } \log n   - \frac{1}{2c_\eps}  \log (\log n)^{\gamma}   \right)^{1/\rho_{\eps}  },
\end{align}
with 
\begin{align}
\gamma = \frac{  1/p ( \gamma_3 -\gamma_1)  }{\rho_{\eps}  }
\end{align}
which, in turn,  implies 
\begin{align}
\sup_{(f_X,f_\eps)\in \mathcal{F}^{\ell,u} (X,\eps, p)}  \Ebig{  \| f_{\eps} - \hat{f}_{\eps, h^*  }  \|_{\lk^2}^2  } =O\left(  n^{-\frac{p-1}{p}  }  (\log  n)^{ \frac{  \frac{p-1}{p}\gamma_1+ \frac{1}{p}\gamma_3  }{\rho_\eps} }    \right). 
\end{align}
\textbf{Discussion:}
 We have not considered  the case where both, the target density and the error density, are supersmooth. Deriving rates of convergence in this framework requires the consideration of  various different subcases, leading to rather tedious and cumbersome calculations. We omit the details and  refer to \cite{Lacour} for a detailed discussion on the subject.

\subsubsection*{Comparison to earlier results}
We have mentioned that the rates of convergence derived above differ substantially from the rate results given in \cite{Li_Vuong}. To illustrate this point, the rates are listed in the table below.

\begin{table}[h!]
 \setlength{\tabcolsep}{10pt}
\renewcommand{\arraystretch}{2.5}
\begin{center}
\begin{tabular}{|l||  c | c |  c|   }
\hline
\  & $c_X=0,\  c_{\eps}=0 $  & $c_X=0,\  c_{\eps}>0$ &$c_X>0,\  c_{\eps}=0 $ \\
\hline 
$\hat{f}_X $ &  $n^{-\frac{2\beta_X -1}{2(1+1/p)\beta_X +2\beta_\eps +1}  }$ & $(\log n)^{-\frac{2\beta_X-1}{\rho_\eps} }$ & $(\log n)^{\gamma} n^{-\frac{p}{p+1} } $  \\
\hline  
 $\hat{f}_X^{LV}$ &  $n^{-\frac{2\beta_X-1}{4\beta_X+6\beta_\eps+4} }$& $(\log n)^{-\frac{2\beta_X-1}{\rho_\eps} }$&$(\log n)^{\gamma} n^{-\frac{1}{3} } $  \\
\hline 
\end{tabular}
\caption{Rates of convergence for estimating the target density}
\end{center}
\end{table}
\begin{table}[h!]
 \setlength{\tabcolsep}{10pt}
\renewcommand{\arraystretch}{2.5}
\begin{center}
\begin{tabular}{|l||  c | c |  c|   }
\hline
\  & $c_X=0,\  c_{\eps}=0 $  & $c_X=0,\  c_{\eps}>0$ &$c_X>0,\  c_{\eps}=0 $ \\
\hline 
$\hat{f}_\eps $ &  $n^{- \frac{2\beta_\eps - 1}{ \frac{2(p+1)}{p-1} \beta_X + \frac{2(p+1)}{p-1}  \beta_{\eps} +\frac{p+1}{p-1} }  }$ & $(\log n)^{\gamma} n^{-\frac{p-1}{p} } $ & $(\log n)^{-\frac{2\beta_\eps -1}{\rho_X} }$  \\
\hline  
 $\hat{f}_\eps^{LV}$ &  $n^{-\frac{2\beta_X- 1}{6\beta_X+6\beta_\eps +4} }$& $(\log n)^{\gamma} n^{-\frac{1}{3} }$&$(\log n)^{-\frac{2\beta_\eps -1}{\rho_X} }$  \\
\hline 
\end{tabular}
\normalsize
\caption{Rates of convergence for estimating the noise density}
\end{center}
\end{table}
We need to be careful about the fact that the rates of convergence given in \cite{Li_Vuong} are derived under the assumption the moments of all orders and even all exponential moments are finite, which compares to $p=\infty$. There is no difference in the rate when an ordinary smooth target with supersmooth noise is being considered. In any other case, the gap in the rate is striking. 

It is interesting to note that the rates of convergence found in the present publication do also differ from the rates which have been found for estimators in the structurally simpler case of panel data with symmetric errors, see \cite{Delaigle_repeated} and \cite{Comte_repeated}. In the table below, $\hat{f}_X^{sym}$ is understood to be the estimator for the symmetric error case, defined according to \cite{Comte_repeated}
and $\eps \in (0,1/2)$ is arbitrary. 
\begin{table}[h!]
 \setlength{\tabcolsep}{10pt}
\renewcommand{\arraystretch}{2.5}
\begin{center}
\begin{tabular}{|l||  c | c |  c|   }
\hline
\  & $c_X=0,\  c_{\eps}=0 $  & $c_X=0,\  c_{\eps}>0$ &$c_X>0,\  c_{\eps}=0 $ \\
\hline 
$\hat{f}_X $ &  $n^{-\frac{2\beta_X -1}{2(1+1/p)\beta_X +2\beta_\eps +1}  }$ & $(\log n)^{-\frac{2\beta_X-1}{\rho_\eps} }$ & $(\log n)^{\gamma} n^{-\frac{p}{p+1} } $  \\
\hline  
 $\hat{f}_X^{sym}$ &  $n^{-\frac{2\beta_X-1}{2(\beta_X\vee \beta_\eps)+2\beta_\eps} }$& $(\log n)^{-\frac{2\beta_X-1}{\rho_\eps} }$&$(\log n)^{\gamma} n^{-1+\eps } $  \\
\hline 
\end{tabular}
\caption{Rates of convergence for estimating $f_X$, symmetric vs. non-symmetric error case.}
\end{center}
\end{table}

The convergence rates coincide if an ordinary smooth target density with supersmooth errors is being considered.  When both, $f_X$ and $f_\eps$ are ordinary smooth  and $\beta_X \geq \beta_\eps$ holds, $\hat{f}^{sym}$ attains the rate $n^{-\frac{2\beta_X - 1}{2\beta_X +2\beta_\eps}}$, which is known to be optimal in deconvolution problems. In this situation, $\hat{f}_X$ shows a slightly worse performance than $\hat{f}_X^{sym}$, which is not surprising in light of the fact that the model with non-symmetric errors has a more complicated structure.  However, it is certainly surprising to notice that for $\beta_\eps >> (1+1/p)\beta_X+1/2$, the rates for $\hat{f}_X$ are substantially better than the rates for $\hat{f}_X^{sym}$. 

\section{Simulation studies}
\subsection{Some data examples} 
For the practical choice of the smoothing parameter, we use  a  leave-p-out cross validation strategy. We consider the parameter set $M=\{1,\cdots, \sqrt{n}\}$ with each parameter $m$ corresponding to the bandwidth $1/m$. 
Given any subset $N:= \{n_1, \cdots, n_p\}\subseteq \{1,\cdots, n\}$ of  size p, we build an estimator $\hat{\phi}_X^{N}$ of  $\phi_X$ based on the subsample  $(Y_k)_{ k\in N}$, as well as an estimator $\hat{\phi}_X^{- N}$ based on   $(Y_k)_{ k\not \in N}$. For $m\in M$, we may  use 
\begin{equation}
\hat{\ell}(\phi_X,\hat{\phi}_{X,1/m}):= \frac{1}{{n \choose p}  }  \sum_{ N=\{n_1,\cdots, n_p \}  }  \| \phi_X^{N}\Fourier \kf_{1/\sqrt{n}}- \phi_X^{-N}\Fourier \kf_{1/m} \|_{\lk^2}^2 
\end{equation}
as an empirical approximation to the loss function 
\begin{equation}
\ell(\phi_X, \hat{\phi}_{X,1/m} )  = \|  \phi_X - \hat{\phi}_{X,1/m} \|_{\lk^2}^2.
\end{equation}
Minimizing the empirical loss leads to selecting 
\begin{equation}
\hat{m}= \argmin\{m\in M: \hat{\ell}(\phi_X, \hat{\phi}_{X,1/m}  ) \}, 
\end{equation}
and working with the bandwidth $\hat{h}= 1/\hat{m}$, thus defining the adaptive estimator $\hat{\phi}_X^{ad} = \hat{\phi}_{X,\hat{h}}$. 

Simulation experiments indicate that the procedure works reasonably well with $p=10$. However, it is evident that even for small  sample sizes, the complexity of the algorithm explodes and the procedure is numerically intractable. To deal with this problem, we use a modified algorithm. We subdivide $\{1,\cdots, n\}$ into $n/5$  disjoint blocks $B_1, \cdots, B_{n/5}$ of size 5 and  build our leave-10-out estimators, based on the subsets $N_k=B_k\cup B_{k+1}, k=1, \cdots, n/5-1$. 

We work with a Gaussian kernel and try the procedure for the following  target densities and  errors: 
\begin{enumerate}
\item[(i)] $X$ has a  $\Gamma(4,2)$ distribution and $\eps$ has a bilateral Gamma distribution with parameters $2,2,3,3$, that is, the corresponding density is the convolution of a $\Gamma(2,2)$-density, supported on $\R_+$ and a $\Gamma(3,3)$-density, supported on $R_{-}$.  In the sequel, we abbreviate this type of distributions by $b\Gamma(2,2,3,3)$. 
\item[(ii)] $X$ has a $b\Gamma(1,1,2,2)$-distribution and the errors are, up to a location shift, $\Gamma(4,2)$-distributed, that is, $\eps+2\sim \Gamma(4,2)$.  In the sequel, we write $\eps\sim \Gamma(4,2)-2$. (The location shift is necessary to ensure that $\E[\eps]=0$ holds true.) 
\item[(iii)]  $X$ has a standard normal distribution, $X \sim  \mathcal{N}(0,1)$  and $\eps  \sim b\Gamma(2,2,3,3)$. 
\item[(iv)] $X$ has again a standard normal distribution. $\eps$ is a mixture of two normal distributions with 
parameters $-2,1$ and $2,2$.  We use the notation $\eps \sim m\mathcal{N}(-2,1,2,2)$. 
\end{enumerate}

We use a Gaussian kernel and run the procedure for $n=100, 1000, 10000$ observations. Based on $500$ repetitions of the adaptive procedure, we calculate the empirical risk $\hat{r}^{ad}$ and compare this quantity to the empirical risk $\hat{r}^{or}$  of the "estimator" with oracle choice of the bandwidth. The values are summarized in the table below. \\
\   \\
\small
\begin{center}
 \setlength{\tabcolsep}{10pt}
\renewcommand{\arraystretch}{2}
\begin{tabular}{|l||  p{0,14\textwidth}|  p{0,14\textwidth} ||  p{0,14\textwidth}|  p{0,14\textwidth}  |   }
\hline
\  & \multicolumn{2}{|c||}{\hspace*{-0,2cm}{$X\sim \Gamma(4,2),\  \eps \sim b\Gamma(2,2,3,3)$ } }  & \multicolumn{2}{|c|}{\   $X\sim \mathcal{N}(0,1),  \eps \sim b\Gamma(2,2,3,3)$} \\
\hline
$n$  & \centering{$\hat{r}^{or}$} & \centering{$ \hat{r}^{ad}$}   & \centering{$\hat{r}^{or}$} &   {$ \hspace*{0,2cm}\hat{r}^{ad}$}  \\ 
\hline
100 & \centering{0.0151}  & \centering{0.0198}   & \centering{0.0104}  &\   0.0198  \\
\hline
1000 & \centering{0.0034}    & \centering{0.0076} &  \centering{0.0019}    &\   0.0044  \\
\hline
10000 & \centering{ 0.0015
}  &   \centering{0.0040} & \centering{0.0007} & \  0,0016 \\
\hline
\end{tabular}
\end{center}
\begin{center}
 \setlength{\tabcolsep}{10pt}
\renewcommand{\arraystretch}{2}
\begin{tabular}{|l||  p{0,14\textwidth}|  p{0,14\textwidth} ||  p{0,14\textwidth}|  p{0,14\textwidth}  |   }
\hline
\  & \multicolumn{2}{|c||}{ \hspace*{-0,3cm} $X\sim \mathcal{N}(0,1),  \eps\sim m\mathcal{N}(\text{-}2,1,2,2)$  }  & \multicolumn{2}{|c|}{ \hspace*{-0,2cm}$X\sim b\Gamma(1,1,2,2),   \eps\sim \Gamma(4,2)\text{-}2$  } \\
\hline
$n$  & \centering{$\hat{r}^{or}$} & \centering{$ \hat{r}^{ad}$}   & \centering{$\hat{r}^{or}$} & \   {$ \hat{r}^{ad}$}  \\ 
\hline
100 & \centering{0.0310}  & \centering{0.0410}   & \centering{ 0.0135}  &\   \  0.0172  \\
\hline
1000 & \centering{0.0118}    & \centering{0.0352} &  \centering{0.0027}    &\    0.0074 \\
\hline
10000 & \centering{0.0040}  & \centering{0.0067} & \centering{0.0013} &\  0.0038\\
\hline
\end{tabular}
\end{center}
\subsection{Comparison to the symmetric error case}
We have mentioned that so far, the model of repeated observations has mainly been studied under the additional assumption that the error terms are symmetric.  In Section 3, it turned out that our rates of convergence are, in some cases, better than the rate results presented in  \cite{Delaigle_repeated} or \cite{Comte_repeated}. So far, it is not clear if this gap in the rate is due to a sub-optimal upper bound in the mentioned papers or to a different performance of the estimators themselves. 

Simulation studies indicate that the estimator which has been designed to handle the case of skew errors does indeed outperform, in some cases, the standard estimator for the symmetric error case. 

Before having a look at some data examples, let us give a brief outline on the estimation strategy for the symmetric error case:  In the panel data model, suppose that $\eps$ has a symmetric distribution. In this case, 
\begin{equation}
Y_{j,1}-Y_{j,2} = \eps_{j,1}- \eps_{j,2}\overset{d}{=}  \eps_{j,1} +\eps_{j,2}, \  j=1, \cdots, n. 
\end{equation}
Consequently, $\phi_{Y_1-Y_2}=\phi_{\eps}^2$. An unbiased estimator of $\phi_{\eps}^2$ can then be built from the data set $(Y_{j,1}- Y_{j,2})_{j=1,\cdots, n}$. Taking square roots gives an estimator $\hat{\phi}_{\eps}$  of $\phi_{\eps}$ and a regularized version of this estimator is plugged in the denominator.  Again $\phi_Y$ can be estimated directly from the data.  For the details, we refer to \cite{Comte_repeated}.
In the sequel, we denote by $\hat{\phi}_X^{sym}$ the estimator for the symmetric error case. 

As indicated by the theory, it turns out that $\hat{\phi}_X$ performs substantially better than $\hat{\phi}_X^{sym}$ if the  error density is very smooth, in comparison to the target density. To illustrate this phenomenon, we have a look at the following examples: 
\begin{itemize}
\item[(i)] $X$ has a $\Gamma(2,4)$-distribution and $\eps \sim b\Gamma(3,2,3,2)$. 
\item[(ii)] $X\sim b\Gamma(1,2,1,2)$ an $\eps\sim b\Gamma(4,3,4,3)$. 
\end{itemize}
When we consider  target densities  which are very smooth, in comparison to the error density,  the estimator discussed in the present paper does, on small or medium sample sizes,  still perform slightly better than the estimator for the symmetric error case.  However, this difference in the performance is small and vanishes completely as the sample size increases.   For illustration, we consider the following  examples:
\begin{itemize}
\item[(iii)] $X\sim b\Gamma(4,3,4,3)$ and $\eps\sim b\Gamma(1,2,1,2)$. 
\item[(iv)] $X\sim \mathcal{N}(0,1)$ and $\eps \sim b\Gamma(3,5,3,5)$. 
\end{itemize}

In the table below, based on $500$ repetitions of the estimation procedure (with oracle choice of the bandwidth), we compare the empirical risk
$\hat{r}^{or}$  of $\hat{\phi}_X$ to the empirical risk  $\hat{r}^{sym, or}$ of $\hat{\phi}_X^{or}$. 
\   \\
\begin{center}
 \setlength{\tabcolsep}{10pt}
\renewcommand{\arraystretch}{2}
\begin{tabular}{|l||  p{0,14\textwidth}|  p{0,14\textwidth} ||  p{0,14\textwidth}|  p{0,14\textwidth}  |   }
\hline
\  & \multicolumn{2}{|c||}{$X\sim \Gamma(2,4),\  \eps \sim  b\Gamma(3,5,3,5)$  }  & \multicolumn{2}{|c|}{\   $X\sim b\Gamma(1,2,1,2)\  \eps \sim b\Gamma(4,3,4,3)$} \\
\hline
$n$  & \centering{$\hat{r}^{ or}$} & \centering{$ \hat{r}^{sym, or}$}   & \centering{$\hat{r}^{or}$} & \hspace*{0,4cm }  {$ \hat{r}^{sym, or}$}  \\ 
\hline
100 & \centering{0.09721}  & \centering{0.25311}   & \centering{0.04373}  & {0.12089}   \\
\hline
1000 & \centering{0.05917}    & \centering{0.15747} &  \centering{0.02442}    &{0.08018}   \\
\hline
10000 & \centering{0.03955 
}  & \centering{0.10378} & \centering{0.01491} &{0.05250} \\
\hline
\hline
\  & \multicolumn{2}{|c||}{ \   $X\sim b\Gamma(4,3,4,3), \  \eps\sim b\Gamma(1,2,1,2)$  }  & \multicolumn{2}{|c|}{\ $X\sim\mathcal{N}(0,1), \  \eps\sim b\Gamma(3,5,3,5)$  } \\
\hline
n  & \centering{$\hat{r}^{or}$} & \centering{$ \hat{r}^{sym,or}$}   & \centering{$\hat{r}^{or}$} & \hspace*{0,4cm }  {$ \hat{r}^{sym,or}$}  \\ 
\hline
100 & \centering{0.00930}  & \centering{0.01446}   & \centering{0.00631}  &\ {0.00903}   \\
\hline
1000 & \centering{0.0032}    & \centering{0.00424} &  \centering{0.00184}    &\    {0.00254}   \\
\hline
10000 & \centering{0.00057}  & \centering{0.00058} & \centering{0.00077} & \  {0.00071} \\
\hline
\end{tabular}

\end{center}
\textbf{Conclusion:} Our simulation studies indicate that our estimator is, in some cases, preferable to the estimation procedures designed for the symmetric error case. In other cases, the performance of both procedures is practically identical. 

However, if the errors are unknown it is clear that in practical applications, one cannot be sure if the symmetry assumption on the errors is satisfied, so we conclude that it is preferable, in either case, to work with the procedure which is designed for the non-symmetric case. 
\section{Proofs}
\subsection{Proof of Theorem \ref{Schranke_X} }
We start by providing some auxiliary results to prepare the proof of Theorem \ref{Schranke_X}.
In the sequel, we use the following short notation.
\begin{align}
& \rem(u) :=  \frac{1}{\psi(u)}-  \frac{1}{\psitilde(u)} ;   \    \  
\hat{c}(u):= \dpsiemp(0,u) - \dpsi(0,u);  \    \     \hat{c}_j(u) := 
i Y_{j,1}e^{iu Y_{j,2} }-  \dpsi(0,u); \\
&     \hat{b}(u):= \tilde{\psi}(0,u)  -\psi(0,u) \    \          \text{and} \   \  
\Psi\rq{}(u_2):= \frac{\partial}{\partial u_1} \log \psi(0,u_2).
\end{align}
Moreover, 
\begin{equation}
\Delta(u):=   \int_0^u \limits \left(   \frac{\dpsiemp(0,u_2)  }{\tilde{\psi}
(0,u_2)   }-   \frac{ \dpsi(0,u_2) }{\psi(0,u_2)   }   \right)  \d u_2.
\end{equation}
First, we consider the deviation of $1/\tilde{\psi}$ from its target: 
\begin{lemma} \label{Rest_psi}It holds that  for some positive constant $C$ depending on $p$,
\begin{align}
\Ebig{ \bigg|\frac{1}{\psi(0,u)}-\frac{1}{\tilde{\psi}(0,u)}    \bigg|^{2p}        } \leq  C \min \Big\{\frac{n^{-p} }{|\psi(0,u)|^{4p} },
\frac{1}{|\psi(0,u)|^{2p} }  \Big\}.
\end{align}
\end{lemma}
\begin{proof}  Consider first the case where $|\psi(0,u)|\geq n^{-1/2}$.  We start by  observing that 
\begin{align}
\Ebig{|\hat{b}(u)|^{2p} }  \leq 4^p  \Big( \Ebig{ |{\psi}(0,u) -\hat{\psi}(0,u)|^{2p}   }   +  \Ebig{ |\hat{\psi}(0,u) -\tilde{\psi}(0,u)|^{2p}   }  \Big).
\end{align}
By Rosenthal's inequality, for some constant $C$ depending on $p$, 
\begin{align}
 \Ebig{ |{\psi}(0,u) -\hat{\psi}(0,u)|^{2p}   }  \leq \frac{C}{n^p}.
\end{align}
Moreover, by definition of $\tilde{\psi}$, 
\begin{align}
\Ebig{  |\hat{\psi}(0,u) -\tilde{\psi}(0,u)|^{2p}                                            } \leq \Ebig{\Big( |\hat{\psi}(0,u)| +n^{-1/2}\Big)^{2p}1_{\{|\hat{\psi}(0,u)|\leq n^{-1/2}\}  } }      \leq 4^p n^{-p }.                                       
\end{align}
Consequently, $\Ebig{ |\hat{b} (u)|^{2p} } \leq C n^{-p}$. 
Now, 
\begin{align}
 \bigg|\frac{1}{\psi(0,u)}-\frac{1}{\tilde{\psi}(0,u)}    \bigg|^{2p}    =  \bigg|\frac{\hat{b}(u)}{\psi(0,u)\tilde{\psi}(0,u)}   \bigg|^{2p}    \leq 4^p  \Big(  \frac{|\hat{b}(u)|^{2p}  }{|\psi(0,u)|^{4p} }   +   {\frac{{|\hat{b}(u)|^{4p} } }{|\psi(0,u)|^{4p} |\tilde{\psi}(0,u)|^{2p}  }  }\Big). 
\end{align}
We have 
\begin{align}
\Ebig{ \frac{|\hat{b}(u)|^{2p}  }{|\psi(0,u)|^{4p} }   }  \leq C \frac{n^{-p} }{|\psi(0,u)|^{4p} } 
\end{align}
and, since  $1/|\psi(0,u)| \leq \sqrt{n}$ by definition,
\begin{align}
\Ebig{ \frac{{|\hat{b}(u)|^{4p} } }{|\psi(0,u)|^{4p} |\tilde{\psi}(0,u)|^{2p}  }               }  \leq C \frac{n^{-p} }{|\psi(0,u)|^{4p}   }. 
\end{align}
On the other hand, for $|\psi(0,u)|\leq n^{-1/2}$,  we have the series of inequalities 
\begin{align}
 & \Ebig{\Big|\frac{1}{\psi(0,u)  }-\frac{1}{\tilde{\psi}(0,u)}     \Big|^{2p}  }
\leq 4^p \Big( \frac{1}{|\psi(0,u)|^{2p}  }  +   \Ebig{ \frac{1}{|\tilde{\psi}(0,u)|^{2p}  }    }   \Big)   \leq 
4^p \Big( \frac{1}{|\psi(0,u)|^{2p}  }  + n^p  \Big)   \\
\leq &   4^p \Big( \frac{2}{|\psi(0,u)|^{2p}  }   \Big).
\end{align}
This completes the proof. 
\end{proof}
The following result gives control on  $\Delta$: 
\begin{lemma} \label{Lemma_Delta} Assume that $\E[|Y_1|^{2p} ]< \infty$. Then for some positive constant $C$, 
\begin{align}
\ebig{  |\Delta(u)| 1_{\{  |\Delta(u)| >1  \}   }    }
\leq  CG(X,\eps, u,p)  \frac{1}{n^p}  \bigg(   \intu   \frac{1}{|\psi(0,u_2)|^2  }  \d u_2   \bigg) ^p
\end{align}
with
\begin{align}
&G(X,\eps, u,p) \\
 =  &   ( \| \phi''_X \phi_\eps     +\E[\eps^2 ] \phi_X \phi_\eps \|_{\lk^1}   +  \| \phi_X'\phi_\eps \|_{\lk^2} ^2  )^p
+ \Big( \intu  |\Psi'(x) |^2  \d x \Big)^p  +
\frac{u^p 1_{\{p\geq 2 \}}  \E[|Y_1|^{2p}  ]  }{n^{p-1}  }  +\E^{1/2}[|Y_1|^{2p}  ].
\end{align}
Moreover  
\begin{align}
\ebig{  |\Delta(u)|^{2p} 1_{\{  |\Delta(u)| \leq 1  \}   }    }
\leq C G(X,\eps, u,p)    \frac{1}{n^p}  \bigg( \intu \frac{1}{|\psi(0,u_2)|^2 }  \d u_2\bigg)^p  .
\end{align}
\end{lemma}
\begin{proof}
We can estimate 
\begin{align}
 \big| \Delta(u)\big|  = &   \bigg|  \int_0^u \limits  \left( \frac{\dpsiemp(0,u_2)}{\psitilde(0,u_2) }
-\frac{\dpsi(0,u_2)}{\psi(0,u_2) }   \right) \d u_2  \bigg|  
\leq   \bigg|  \int_0^u \limits  \left( \frac{\dpsiemp(0,u_2)-\dpsi(0,u_2)}{\psi(0,u_2) }  \right) \d u_2  \bigg|\\
&+  \bigg|  \int_0^u \limits  \dpsi(0,u_2)\rem(u_2)   \d u_2  \bigg|  + \bigg|  \int_0^u \limits  \left( \dpsiemp(0,u_2)-\dpsi(0,u_2) \right)  \rem(u_2)   \d u_2   \bigg|  \\
  =:&  \Delta_1(u) +\Delta_2(u)  +\Delta_3(u). 
\end{align}
Rosenthal's inequality (see, for example, \cite{Ibragimov_Ros} )  gives for some constant $C$ depending only on $p$: 
\begin{align}
\Ebig{ \Delta_1(u)^{2p}  }= & \Ebig{  \bigg|  \int_0^u \limits  \left( \frac{\dpsiemp(0,u_2)-\dpsi(0,u_2)}{\psi(0,u_2) }  \right) \d u_2  \bigg|^{2p}                     }
=  \Ebig{ \bigg|\frac{1}{n}  \sum_{j=1}^n \int_0^u\limits  \frac{  \hat{c}_j(u_2)        }{ \psi(0,u_2)    }   \d u_2  \bigg|^{2p}   }   \\
&\leq   C   ( \frac{1}{n^{p}}     \bigg(  \Ebig{  \bigg| \int_0^u \limits   
\frac{\hat{c}_j(u_2)}{\psi(0,u_2) }  \d u_2\bigg|^2  }  \bigg) ^{p}
+    \frac{1}{n^{2p-1}  }   \Ebig{  \bigg| \int_0^u \limits   
\frac{\hat{c}_j(u_2)}{\psi(0,u_2) }  \d u_2\bigg|^{2p}  }          ) .
\end{align}
Using Fubini's theorem, the Cauchy-Schwarz inequality and Lemma \ref{Lemma_Momente}, we derive that 
\begin{align}
&  \Ebig{  \bigg| \int_0^u \limits   
\frac{\hat{c}_j(u_2)}{\psi(0,u_2) }  \d u_2\bigg|^2  }  = \int_0^u \limits \int_0^u \limits \frac{\Cov( \hat{c}_j(x),\hat{c}_j(y))}{\psi(0,x) \psi(0,\m y) }  \d x \d y\\
=&    \int_0^u \limits \int_0^u \limits \frac{\E[(iY_1)^2 e^{i(x-y)Y_2} ]}{\psi(0,x) \psi(0,\m y) }  \d x \d y
-\int_0^u \limits \int_0^u \limits \frac{\E[iY_1e^{ixY_2} ]\E[iY_1e^{\m iyY_2} ]}{\psi(0,x) \psi(0,\m y) }  \d x \d y  \\
\leq & \int_0^u \limits \int_0^u \limits \frac{|\E[(iY_1)^2 e^{i(x-y)Y_2} ]|}{|\psi(0,x)|^2}  \d x \d y
\leq  \sup_{x\in [0,u] } \limits   \int_0^u \limits |\E[(iY_1)^2 e^{i ( x-y)  Y_2}| \d y   \int_0^u \limits\frac{ 1}{| \psi(0,y)|^2  } \d y \\
\label{Zerlegung_2}\leq  & \left( \|\phi_X''\phi_{\eps} +\E[\eps^2] \phi_X\phi_{\eps} \|_{\lk^1}    \right) \int_0^u\limits\frac{ 1}{| \psi(0,y)|^2  } \d y.
\end{align}
For $p\geq 2$, the Cauchy-Schwarz inequality gives 
\begin{align}
&  \frac{1}{n^{2p-1}  }   \Ebig{  \bigg| \int_0^u \limits   
\frac{\hat{c}_j(u_2)}{\psi(0,u_2) }  \d u_2\bigg|^{2p}  }    
\leq  \frac{1}{n^{2p-1}  } \bigg( \int_0^u \limits \frac{1}{|\psi(0,u_2)|^2 } \d u_2 \bigg)^{p}
 \Ebig{ \left( \int_0^u \limits  |\hat{c}_j|^2   \d u_2 \right)^{p}         }\\
\leq  &  \frac{4^p\E[|Y_j|^{2p}]}{n^{2p-1}} u^{p}
 \bigg( \int_0^u \limits \frac{1}{|\psi(0,u_2)|^2 } \d u_2 \bigg)^{p}.
\end{align}   
We have thus shown 
\begin{align}
\Ebig{ \Delta_1(u)^{2p}   }   \leq    C G(X,\eps,p,u )
\bigg(    \frac{1}{n}  \intu \frac{1}{|\psi(0,u_2) |^2  }  \d u_2 \bigg)^p . 
\end{align}
Next, thanks to Lemma  \ref{Rest_psi} and the H\"older inequality,
\begin{align}
\Ebig{\Delta_2(u)^{2p}  } =&   \Ebig{  \bigg|\int_0^u\limits  \dpsi(0,u_2)  \rem(u_2)  \d u_2      \bigg|^{2p}} = \Ebig{  \bigg|\int_0^u\limits \Psi'(u_2) \psi(0,u_2)  \rem(u_2)   \bigg|^{2p}} \\
\leq & \bigg( \intu  |\Psi'(x) |^2 \d x \bigg)^p \Ebig{ \bigg(\intu |\psi(0,u_2)|^2
|\rem(u_2) |^2\d u_2    \bigg)^p } \\
 \leq &\bigg( \intu  |\Psi'(x) |^2 \d x \bigg)^p \bigg( \intu \frac{1}{|\psi(0,u_2)|^2}  \d u_2  \bigg)^{p-1}
\intu  \frac{1}{|\psi(0,u_2)|^2} |\psi(0,u_2) |^{4p} \Ebig{|\rem(u_2) |^{2p} } \d u_2 \\
\leq &\frac{C}{n^p}   \bigg( \intu  |\Psi'(x) |^2 \d x \bigg)^p \bigg( \intu \frac{1}{|\psi(0,u_2)|^2}  \d u_2  \bigg)^{p} \leq    C G(X,\eps,p,u )
\bigg(    \frac{1}{n}  \intu \frac{1}{|\psi(0,u_2) |^2  }  \d u_2 \bigg)^p.
\end{align}
Finally,  another application of Lemma \ref{Rest_psi} and the H\"older inequality  gives 
\begin{align}
\Ebig{\Delta_3(u)^p} = &\Ebig{  \bigg|  \int_0^u \limits  \left( \dpsiemp(0,u_2)-\dpsi(0,u_2) \right)  \rem(u_2)  \d u_2   \bigg|^{p} }
= \Ebig{  \bigg|  \int_0^u \limits  
 \hat{c}(u_2)   \rem(u_2)   \d u_2   \bigg|^{p} }  \\
\leq &   \bigg(    \intu  \frac{1}{|\psi(0,u_2) |^2  }  \d u_2  \bigg)^{ p-1}  \intu  \frac{ |\psi(0,u_2)|^{2p}\Ebig{|\hat{c}(u_2) \rem(u_2) |^{p}   } }{  |\psi(0,u_2) |^2   }  \d u_2  \\
\leq & \bigg(    \intu  \frac{1}{|\psi(0,u_2) |^2  }  \d u_2  \bigg)^{ p-1} 
\intu  \frac{ |\psi(0,u_2)|^{2p}\E^{\frac{1}{2} }\Big[ |\hat{c}(u_2)|^{2p} \Big]  \E^{\frac{1}{2}  } \Big[| \rem(u_2) |^{2p}   \Big] }{  |\psi(0,u_2) |^2   }  \d u_2  \\
\leq &  C \frac{\E^{\frac{1}{2}  }[ |Y_1|^{2p}  ]  }{n^{p}  }\bigg(    \intu  \frac{1}{|\psi(0,u_2) |^2  }  \d u_2  \bigg)^{ p}
\leq C  G(X,\eps, p,u)   \bigg(  \frac{1}{n}  \intu  \frac{1}{|\psi(0,u_2) |^2  }  \d u_2  \bigg)^{ p}.
\end{align}
We set  
\[
A_j=A_j(u):= \{ |\Delta(u)|>1  \}  \cap \{  \argmax_{k=1,2,3}|\Delta_k(u)| =j  \}.
\]
We may use the fact that on $A_j$, $\Delta(u) \leq 3 \Delta_j(u)$ as well as $\Delta_j(u)>1/3$, to conclude that 
\begin{align}
 \ebig{  |\Delta(u)| 1_{\{  |\Delta(u)| >1  \}   }    }
& \leq   3\Big( \ebig{  |\Delta_1(u)| 1_{A_1} }  +\ebig{  |\Delta_2(u)| 1_{A_2} }    + \ebig{  |\Delta_3(u)| 1_{A_3} }   \Big)   \\
& \leq 3^{2p}   \Big( \ebig{  |\Delta_1(u)|^{2p}  } + 
\ebig{  |\Delta_2(u)|^{2p}  }  +  \ebig{  |\Delta_3(u)|^{p}  }   \Big) .
\end{align}
Combining this inequality with the moment  bounds on the $\Delta_j(u)$, we have shown that  for a constant $C$ depending only on $p$, 
\begin{align}
\Ebig{  |\Delta(u)|  1_{ \{|\Delta(u)| >1 \}  }   }
\leq C  G(X,\eps, p,u)    \frac{1}{n^p}  \bigg(  
\frac{1}{|\psi(0,u_2)|^2 } \d u_2   \bigg)^p. 
\end{align} 
Next, we define 
\[
B_j:= \{ |\Delta(u)| \leq 1 \}  \cap  \{   \max\{ \Delta_k(u)|k=1,2,3\}  = j \}, \   j=1,2,3.
\]
It holds that 
\begin{align}
\Ebig{ |\Delta(u)|^{2p} 1_{\{|\Delta(u)|  \leq 1 \}  }  }
& \leq 9^p \Big(  \E[\Delta_1(u)^{2p} 1_{B_1}  ]  +\E[\Delta_2(u)^{2p} 1_{B_2}  ]      +\Ebig{ |\Delta(u)|^{p}  1_{B_3}  }\Big)   \\
&  \leq 9^p \Big(  \E[\Delta_1(u)^{2p} 1_{B_1}  ]  +\E[\Delta_2(u)^{2p} 1_{B_2}  ]       +  \E[\Delta_3(u)^p 1_{B_3}  ] \Big)  .
\end{align}
This implies, using again  the moment bounds on the $\Delta_j$, 
\begin{align}
\Ebig{ |\Delta(u)|^{2p} 1_{\{|\Delta(u)|  \leq 1 \}  }  }
\leq C G(X,\eps , p,u)    \bigg( \frac{1}{n}  \intu  \frac{1}{|\psi(0,u_2)|^2 }\d u_2  \bigg)^p . 
\end{align}
\end{proof}
We can now prove the upper bounds for $\hat{f}_{X,h}$: 
\begin{proof}[\textbf{Proof of Theorem \ref{Schranke_X}}] Parseval's identity gives 
\begin{align}
\Ebig{  \big\|  f- \hat{f}_h \big\|_{\lk^2}^2 } \leq  
2 \big\|  f- \kf_h\ast f  \big\|_{\lk^2}^2
+ \frac{1}{\pi}  \int  |\Fourier \kf_h(u)|^2 \Ebig{ |\phi_X(u) - \hat{\phi}_X(u)|^2  } \d u.
\end{align}
We  use the trivial observation that $|\phi_X(u)- \hat{\phi}_X(u) |\leq 
|\phi_X(u)- \ecfmod (u)|$, as well as the fact that for $z\in \C$ with $|z|\leq 1$,  $|1- \exp(z)|\leq 2 |z|$ holds,  to derive that 
\begin{align}
& |\phi_X(u)- \hat{\phi}_X(u) |^2 1_{\{|\Delta(u)| \leq 1\}  }
\leq |\phi_X(u)- \ecfmod(u) |^2 1_{\{|\Delta(u)| \leq 1\}  }  \\
=&   |\phi_X(u)(1- \ecfmod(u)/\phi_X(u) )|^2 1_{\{|\Delta(u)| \leq 1\}  }
= |\phi_X(u)(1- \exp(\Delta(u) )   |^2 1_{\{|\Delta(u)| \leq 1\}  }\\
  \leq &   2|\phi_X(u)\Delta(u)|^2  1_{\{|\Delta(u)| \leq 1\}  }.
\end{align}
On the other hand, using the fact that $|\phi_X(u) - \hat{\phi}_X(u)|\leq 2$, as well as the Markov-inequality, we can estimate   
\begin{align}
|\phi_X(u) - \hat{\phi}_X(u)|^{2p} 1_{\{|\Delta(u)| > 1\}  }  \leq 
4^p |\Delta(u)| 1_{\{|\Delta(u)| > 1 \}  } .
\end{align}
Lemma \ref{Rest_psi},  Lemma  \ref{Lemma_Delta} and (A4)  thus give 
\begin{align}
& \Ebig{ |\phi_X(u) - \hat{\phi}_X(u) | ^2 }  \leq  2 |\phi_X(u)|^2 \ebig{|\Delta(u)|^2 1_{\{|\Delta(u)|\leq 1\} } }    +  4\ebig{  |\Delta(u)| 1_{\{|\Delta(u)|>1 \}  }   }  \\
\leq &    \frac{ C G(X,\eps,1,u) }{n}   |\phi_X(u)|^2    \intu \frac{1}{|\psi(0,u_2) |^2 }  \d u_2      + C G(X,\eps, p,u)  \bigg(  \frac{1}{n}  \intu \frac{1}{|\psi(0,u_2) |  }\d u_2  \bigg)^p  \\
  \leq &  \frac{C C_X G(X,\eps,1,u) }{n}  \intu \frac{1}{|\phi_{\eps}(u_2)|^2} \d u_2  +   C G(X,\eps,p,u) \bigg(  \frac{1}{n}  \intu \frac{1}{|\psi(0,u_2) |  }\d u_2 \bigg)^p.
\end{align}
Hence,  by assumption on the support of $\kf$, 
\begin{align}
&\inth  |\kf_h(u)|^2  \Ebig{ |\phi_X(u) - \hat{\phi}_X(u) | ^2 }  \d u \\
\leq & \frac{  CC_XG(X,\eps,1,1/h)}{n}  \inth  \intu  \frac{1}{|\phi_{\eps}(z)|^2 }  \d z  \d u  +
 \frac{ C G(X,\eps,p,1/h)}{n^p}  \inth \bigg(  \intu \frac{1}{|\psi(0,u_2) |  }\d u_2 \bigg)^p  \d u.
\end{align}
This completes the proof. 
\end{proof}
\subsection{Proof of Theorem  \ref{Schranke_eps}  }
\begin{lemma} \label{Lemma_Rest} Let $q\geq p$. Assume that $\E[|Y_1|^{2q} ]<\infty$. Then for some positive constant $C$ depending only on $p$ and $q$, 
\begin{align}
& \Ebig{    \bigg| \frac{1}{\phi_X(u)}-  \frac{1}{\tilde{\phi}_X(u) }   \bigg|^{2p}                       }\\
\leq &  C  \Big[ 
\frac{G(X,\eps,p,u)}{|\phi_X(u)|^{2p}  } \bigg(  \frac{1}{n} \int_0^u\limits 
 \frac{1}{|\psi(0,u_2) |^2}  \d u_2   \bigg)^{p} \hspace*{-0,2cm}
+ \frac{  G(X,\eps,q,u)}{|\phi_X(u)|^{4p} n^{q-p}  }   \bigg(  \int_0^u\limits 
 \frac{1}{|\psi(0,u_2) |^2}  \d u_2   \bigg)^{q} \hspace*{-0,2cm}+\frac{1}{n^p |\phi_X(u)|^{4p} }
\Big].  
\end{align}
\end{lemma}
\begin{proof}
We have 
\begin{align}
 \Ebig{    \bigg| \frac{1}{\phi_X(u)}-  \frac{1}{\tilde{\phi}_X(u) }   \bigg|^{2p}                       }  = \Ebig{    \frac{|\phi_X(u) - \tilde{\phi}_X(u)|^{2p}}{|\phi_X(u)  \tilde{\phi}_X(u)|^{2p}}                   } . 
\end{align}
Using the definition  of $\ecfmod$, as well as the fact that $|\exp(z)|\geq 1/e$ holds for $z\in \C$,  $|z|\leq 1$,  we derive that 
\begin{align}
|\ecfmod (u)|  1_{\{|\Delta(u)| \leq 1 \}  }
= |{\phi}_X(u)| |\exp( \Delta(u) )|1_{\{|\Delta(u)| \leq 1 \}  }  \geq 1/e
|\phi_X(u)| 1_{\{|\Delta(u)| \leq 1\}  }.
\end{align}
Consequently,  by definition of $\tilde{\phi}_X$ and $\hat{\phi}_X$, 
\begin{align}
&|\tilde{\phi}_X(u) |1_{\{|\Delta(u)| \leq 1\}  } 
\geq      |\hat{\phi}_X(u)| 1_{\{|\Delta(u)| \leq 1\}  }  = \Big(   |\ecfmod(u)| 1_{\{|\ecfmod(u)| \leq 1 \}  } 
+1_{\{|\ecfmod(u)| \geq 1\} } \Big)  1_{\{|\Delta(u)| \leq 1 \}  } \\
\geq & \frac{1}{e} |\phi_X(u)|1_{\{|\Delta(u)|\leq 1\}  }.
\end{align}
Next, 
\begin{align}
    |\phi_X(u) - \tilde{\phi}_X(u)|^{2p}           \leq 4^{p}
 \Big( |\phi_X(u) - \hat{\phi}_X(u) |^{2p} +       |\hat{\phi}_X(u) -\tilde{\phi}_X(u)|^{2p} \Big) 
\end{align}
and it holds that 
\begin{align}
\Ebig{     |\hat{\phi}_X(u) -\tilde{\phi}_X(u)|^{2p}                 } &= \Ebig{     |\hat{\phi}_X(u) -\tilde{\phi}_X(u)|^{2p}   1_{\{|\hat{\phi}_X(u)|\leq n^{-1/2}  \} }              }   \\
&\leq  \Ebig{    ( |\hat{\phi}_X(u)| + n^{-1/2} )^{2p}   1_{\{|\hat{\phi}_X(u)|\leq
n^{-1/2}  \}  }              }  \leq 4^p n^{-p}. 
\end{align}
We use Lemma \ref{Lemma_Delta}  
 to conclude that 
\begin{align}
& \Ebig{    \frac{|\phi_X(u) - \tilde{\phi}_X(u)|^{2p}}{|\phi_X(u)  \tilde{\phi}_X(u)|^{2p}}    1_{  \{  |\Delta(u)|\leq 1\}  }               } 
\leq  \frac{\Ebig{|\phi_X(u) - \tilde{\phi}_X(u)|^{2p}  1_{  \{  |\Delta(u)|\leq 1\}  }      }    }{1/e^{2p} |\phi_X(u)  |^{4p}}  \\
 \leq & \frac{2|\phi_X(u)|^{2p} \E[|\Delta(u)|^{2p}1_{\{|\Delta(u)| \leq 1 \}  } ] +n^{-p} }{1/e^{2p}|\phi_X(u)|^{4p} }   \\
\leq & C\Big(  \frac{ G(X,\eps, p,u)  }{  |\phi_X(u)|^{2p} }   \bigg( \frac{1}{n} \int_0^u\limits \frac{1}{|\psi(0,u_2) |^2 } \d u_2    \bigg)^p  +\frac{1}{
n^{p}|\phi_X(u)|^{4p}   }  \Big). 
\end{align}
Next, using  the fact that by definition of $\tilde{\phi}_X$,  $|1/\tilde{\phi}_X(u)|\leq \sqrt{n}$ holds, as well as the fact that 
\begin{align}
\left|  \frac{1}{\tilde{\phi}_X(u)}  \right| \leq 2 \left|  \frac{1}{\phi_X(u) }  \right| +2  \left|\frac{1}{\tilde{\phi}_X(u)  } -  \frac{1}{\phi_X(u) }  \right| =2 \left|  \frac{1}{\phi_X(u) }  \right| +2  
\frac{  |\phi_X(u) - \tilde{\phi}_X(u)  | }{   |\phi_X(u) \tilde{\phi}_X(u) |  }, 
\end{align}
 we conclude that 
\begin{align}
& \Ebig{    \frac{|\phi_X(u) - \tilde{\phi}_X(u)|^{2p}}{|\phi_X(u)  \tilde{\phi}_X(u)|^{2p}}    1_{  \{  |\Delta(u)|> 1\}  }        }  \\
\leq & 4^{p}  \Big(      \frac{ \Ebig{  |\phi_X(u) - \tilde{\phi}_X(u)|^{2p}1_{  \{  |\Delta(u)|> 1\}  }     }    }{|\phi_X(u)  |^{4p}}    +n^p 
      \frac{ \Ebig{  |\phi_X(u) - \tilde{\phi}_X(u)|^{4p}1_{  \{  |\Delta(u)|> 1\}  }     }    }{|\phi_X(u)  |^{4p}}   \Big) 
\\
\leq  &8^p \Big(      \frac{ \Ebig{  |\Delta(u)|  1_{  \{  |\Delta(u)|> 1\}  }  }  +n^{-p}      }{|\phi_X(u)  |^{4p}}    +n^p 
      \frac{ \Ebig{ |\Delta(u)|  1_{  \{  |\Delta(u)|> 1\}  }  } +n^{-2p}   }{|\phi_X(u)  |^{4p}}   \Big)  \\
\leq  &   \frac{C G(X,\eps, q,u)}{|\phi_X(u)|^{4p} }\Big[n^{p}
 \bigg(\frac{1}{n} \int_0^{u}\limits \frac{1}{|\psi(0,u_2)|^2} \d u_2    \bigg)^q 
+\frac{1}{n^p}  \Big]. 
\end{align}
This completes the proof of the lemma. 
\end{proof} 
\begin{proof}[\textbf{Proof of Theorem \ref{Schranke_eps} }] By Parseval's inequality and by assumption on the support of $\kf$, 
\begin{align}
\Ebig{  \|  f_{\eps} - \hat{f}_{\eps,h }   \|_{\lk^2}^2   }
\leq 2 \| f_{\eps} - \kf_h \ast f_{\eps}  \|_{\lk^2}^2 +  \frac{1}{\pi}
\inth \Ebig{  | \phi_{\eps} (u) -\hat{\phi}_{\eps} (u) |^2 }  \d u.
\end{align}

It holds that 
\begin{align}
& |\phi_{\eps}(u) - \hat{\phi}_{\eps}(u)  |^2
= \bigg|\frac{\psi(0,u)}{\phi_X(u) } - \frac{\hat{\psi}(0,u)}{\tilde{\phi}_X(u) } \bigg|^2\\
\leq & 3 \Big( \frac{  |\psi(0,u) - \psiemp(0,u)|^2 }{ |\phi_X(u)|^2 }  
+ |\psi(0,u) - \psiemp(0,u)|^2  \bigg|  \frac{1}{\phi_X(u) }- \frac{1}{\tilde{\phi}_X(u) }  \bigg|^2  +|\psi(0,u)|^2   \bigg|  \frac{1}{\phi_X(u) }- \frac{1}{\tilde{\phi}_X(u) }  \bigg|^2 \Big).
\end{align}  
Since $\psi(0,u)$ is a characteristic function and $\hat{\psi}(0,u)$ its empirical counterpart,  
\begin{align}
\frac{  \E[|\psi(0,u) - \psiemp(0,u)|^2 ]}{ |\phi_X(u)|^2 } \leq n^{-1}
\frac{ 1}{ |\phi_X(u)|^2 }.
\end{align}
 Lemma \ref{Lemma_Rest} and assumption (A5) yield  
\begin{align}
& |\psi(0,u)| ^2\Ebigg{  \bigg|  \frac{1}{\phi_X(u) }- \frac{1}{\tilde{\phi}_X(u) }  \bigg|^2                 } \\
 \leq&C  |\psi(0,u)| ^2 \Big[   \frac{G(X,\eps,1,u) }{|\phi_X(u)|^2 }
\frac{1}{n}  \int_0^u \limits \frac{1}{|\psi(0,z) |^2 }  \d z 
+ \frac{   G(X,\eps, q, u)    }{  |\phi_X(u)|^{4} n^{q-1}  } 
\bigg(  \int_0^u \limits \frac{1}{|\psi(0,z)|^2 }  \d z     \bigg)^{q} 
+\frac{n^{-1}}{|\phi_X(u)|^4 }  \Big]  \\
\leq & C |\phi_{\eps}(u)|^2 \Big[ \frac{G(X,\eps,1,u) }{n}  \int_0^u \limits \frac{1}{|\psi(0,z) |^2 }  \d z +   \frac{G(X,\eps, q, u) }{|\phi_X(u) |^2 n^{q-1} }\bigg(  \int_0^u \limits \frac{1}{|\psi(0,z)|^2 }  \d z \bigg)^{q}   +\frac{ 1 }  {n|\phi_X(u)|^2 } \Big] \\
\leq & C C_{\eps} \Big[ \frac{G(X,\eps,1,u)}{n}   \int_0^u\limits \frac{1}{|\phi_X(z)|^2 }  \d z    +  \frac{ G(X,\eps,1,u)}{ |\phi_X(u)|^2 n^{q-1} }  \bigg(  \int_0^u\limits \frac{1}{|\phi_X(z)|^2 }  \d z \bigg) 
\bigg(\int_0^u\limits \frac{1}{|\psi(0,z)|^2}  \d z \bigg)^{q-1} \\
& +\frac{ 1 }  {n|\phi_X(u)|^2 }   \Big] .
\end{align}
Finally, using the Cauchy-Schwarz inequality and again Lemma  \ref{Lemma_Rest}, we derive that 
\begin{align}
& \ebigg{ |\psi(0,u) - \hat{\psi}(0,u) |^{2}\bigg|  \frac{1}{\phi_X(u) }- \frac{1}{\tilde{\phi}_X(u) }  \bigg|^2                 }
\leq  \E^{\frac{1}{2}  } \Big[ |\psi(0,u) - \hat{\psi}(0,u) |^{4}  \Big]
 \E^{\frac{1}{2} } \Big[ \bigg|  \frac{1}{\phi_X(u) }- \frac{1}{\tilde{\phi}_X(u) }  \bigg|^4   \Big]    \\
\leq & \frac{C}{n}  
\Big[  \frac{G(X,\eps,2,u)^{\frac{1}{2} } }{|\phi_X(u)|^2 n }  \int_0^u  \limits
\frac{1}{|\psi(0,u_2) |^2 } \d u_2
+\frac{
G(X,\eps,2q,(u) )^{\frac{1}{2}} }{|\phi_X(u)|^4n^{q-1}} \bigg(
 \int_0^u\limits \frac{1}{|\psi(0,u_2)|^2}  \d u_2   \bigg)^q  +\frac{1}{n|\phi_X(u)|^4 }  \Big].
\end{align}
Putting the above together, we have shown that for some positive constant $C$,
\begin{align}
&\inth  \Ebig{   \Big|  \phi_{\eps}(u) - \hat{\phi}_{\eps} (u) \Big|^2 }  \ d u
\leq CC_{\eps} \Big[    \frac{ G(X,\eps,1,1/h)}{n }  \inth \intu \frac{1}{|\phi_X(z)|^2 }  \d z \d u  \\
&+  \frac{G(X,\eps,q,1/h)}{n^{q-1}  } \inth \frac{1}{|\phi_X(u)|^2 }  \bigg( \intu \frac{1}{|\phi_X(z )|^2 }  \d z    \bigg) \bigg(  \intu  \frac{1}{|\psi(0,z)|^2 }   \d z  \bigg)^{q-1}  \d u     \\
&  +  \frac{G(X,\eps,2,1/h)^{1/2}}{n^2} \inth  \frac{1}{|\phi_X(u)|^2 }  \ 
\bigg(\intu \limits \frac{1}{|\psi(0,z)|^2}  \d z   \bigg) \d u   \\
 & + \frac{G(X,\eps,2q,1/h) ^{1/2} }{n^q} \inth   \frac{1}{|\phi_X(u)|^4  }  
\bigg(\intu   \frac{1}{|\psi(0,z)|^2}  \d z   \bigg)^q +  \frac{1}{n^2} \inth \frac{1}{|\phi_X(u)|^4}  \d u \Big],
\end{align}
which gives the statement of the theorem. 
\end{proof}
\section{Appendix}
\begin{lemma} \label{Lemma_Momente}
The following holds for the partial derivatives of $\psi$:
\begin{equation}
\frac{\partial^k}{\partial u_1^k}  \psi(0,u_2)= \Ebig{(iY_1)^k e^{iu_2 Y_2}        }=\sum_{m=0}^k {k \choose m}  \E[ (i\eps)^{k-m} ] \phi_{\eps}(u_2)
\phi_X^{(k)}(u_2).
\end{equation}
\end{lemma}
\begin{proof} By definition of $\psi$ and by independence of $X, \eps_1$ and $\eps_2$, 
\begin{align}
\phantom{=}& \frac{\partial^k}{\partial u_1^k}  \psi(0,u_2)=\Ebig{\frac{\partial^k}{\partial u_1^k}  e^{iu_1 Y_1+ iu_2 Y_2} \big|_{u_1=0}          }
= \Ebig{ (i Y_1)^k e^{iu_2 Y_2}  }  \\
=& \sum_{m=0}^k {k \choose m}  \Ebig{ (i X)^m (i\eps_1)^{k-m}  e^{iu_2 X}
e^{iu_2 \eps_2} }  
=   \sum_{m=0}^k {k \choose m}  \E[ (i \eps)^{k-m} ] \E[(i X)^m e^{iu_2 X} ] \E[e^{iu_2 \eps} ]  \\ =& \sum_{m=0}^k {k \choose m}  \E[(i\eps)^{k-m} ] \phi_{\eps}(u_2)
\phi_X^{(m)}(u_2).
\end{align}
\end{proof}
\bibliographystyle{plainnat}
\bibliography{literatur}

\begin{thebibliography}{19}
\providecommand{\natexlab}[1]{#1}
\providecommand{\url}[1]{\texttt{#1}}
\expandafter\ifx\csname urlstyle\endcsname\relax
  \providecommand{\doi}[1]{doi: #1}\else
  \providecommand{\doi}{doi: \begingroup \urlstyle{rm}\Url}\fi

\bibitem[Bonhomme and Robin(2010)]{Bonhomme}
St{\'e}phane Bonhomme and Jean-Marc Robin.
\newblock {Generalized nonparametric deconvolution with an application to
  earning dynamics}.
\newblock \emph{Review of Economic Studies, Oxford University Press},
  77\penalty0 (2):\penalty0 491--533, 2010.

\bibitem[Carroll and Hall(1988)]{Caroll_Hall}
Raymond~J. Carroll and Peter Hall.
\newblock {Optimal rates of convergence for deconvolving a density}.
\newblock \emph{{Journal of the American Statistical Association}}, 83\penalty0
  (404):\penalty0 1184--1186, 1988.

\bibitem[Comte et~al.(2006)Comte, Rosenholc, and Taupin]{Comte_Rosenholz}
Fabienne Comte, Yves Rosenholc, and Marie-Luce Taupin.
\newblock {Penalized contrast estimator for adaptive density deconvolution}.
\newblock \emph{Canadian Journal of Statistics}, \penalty0 (34):\penalty0
  431--452, 2006.

\bibitem[Comte et~al.(2014)Comte, Samson, and Stirnemann]{Comte_repeated}
Fabienne Comte, Adeline Samson, and Julien Stirnemann.
\newblock {Deconvolution Estimation of Onset of Pregnancy with Replicate
  Observations}.
\newblock \emph{Scandinavian Journal of Statistics}, 41:\penalty0 325--345,
  2014.

\bibitem[Delaigle et~al.(2008)Delaigle, Hall, and Meister]{Delaigle_repeated}
Aurore Delaigle, Peter Hall, and Alexander Meister.
\newblock {On deconvolution with repeated measurements}.
\newblock \emph{The Annals of Statistics}, 36\penalty0 (2):\penalty0 665--685,
  2008.

\bibitem[Diggle and Hall(1993)]{Diggle_Hall}
Peter~J. Diggle and Peter Hall.
\newblock {A Fourier Approach to Nonparametric Deconvolution of a Density
  Estimate}.
\newblock \emph{Journal of the Royal Statistical Society. Series B},
  55\penalty0 (2):\penalty0 523--531, 1993.

\bibitem[Efromovich(1997)]{Efromovich}
Sam Efromovich.
\newblock {Density estimation for the case of supersmooth measurement errors}.
\newblock \emph{{Journal of the American Statistical Association}},
  92:\penalty0 526--535, 1997.

\bibitem[Fan(1991)]{Fan}
Jianqing Fan.
\newblock On the optimal rates of convergence for nonparametric deconvolution
  problems.
\newblock \emph{The Annals of Statistics}, 19\penalty0 (3):\penalty0
  1257--1272, 1991.

\bibitem[Ibragimov and Sharakhmetov(2002)]{Ibragimov_Ros}
Rustam Ibragimov and Shaturgun Sharakhmetov.
\newblock {The exact constant in the Rosenthal inequality for random variables
  with mean zero}.
\newblock \emph{Theory of Probability and Its Applications}, 46\penalty0
  (1):\penalty0 127--132, 2002.

\bibitem[Johannes(2009)]{Johannes}
Jan Johannes.
\newblock {Deconvolution with unknown error distribution}.
\newblock \emph{{The Annals of Statistics}}, 37\penalty0 (5a):\penalty0
  2301--2323, 2009.

\bibitem[Kappus and Mabon(2013)]{Kappus_Mabon}
Johanna Kappus and Gwena{\"e}lle Mabon.
\newblock {Adaptive density estimation in deconvolution problems with unknown
  error distribution}.
\newblock \textit{Preprint}, hal-00915982; sumbitted, 2013.

\bibitem[Lacour(2006)]{Lacour}
Claire Lacour.
\newblock {Rates of convergence for nonparametric deconvolution}.
\newblock \emph{Comptes rendus de l'acad\'emie des sciences, Math\'ematiques},
  324\penalty0 (11):\penalty0 877--883, 2006.

\bibitem[Li and Vuong(1998)]{Li_Vuong}
Tong Li and Quang Vuong.
\newblock Nonparametric estimation of the measurement error model using
  multiple indicators.
\newblock \emph{Journal of Multivariate Analysis}, 65\penalty0 (2):\penalty0
  139--165, 1998.

\bibitem[Meister(2004)]{Meister_miss}
Alexander Meister.
\newblock {On the effect of misspecifying the error density in a deconvolution
  problem}.
\newblock \emph{Canadian Journal of Statistics}, 32\penalty0 (4):\penalty0
  439–449, 2004.

\bibitem[Neumann(2006)]{Neumann_2}
Michael Neumann.
\newblock {Deconvolution from panel data with unknown error distribution}.
\newblock \emph{{Journal of Multivariate Analysis}}, 98:\penalty0 1955--1968,
  2006.

\bibitem[Neumann(1997)]{Neumann}
Michael~H. Neumann.
\newblock {On the effect of estimating the error density in nonparametric
  deconvolution.}
\newblock \emph{Journal of Nonparametric Statistics}, 7\penalty0 (4):\penalty0
  307--330, 1997.

\bibitem[Pensky and Vidakovic(1999)]{Pensky}
Marianna Pensky and Brani Vidakovic.
\newblock {Adaptive wavelet estimator for nonparametric density deconvolution}.
\newblock \emph{The Annals of Statistics}, 27\penalty0 (6):\penalty0
  2033--2053, 1999.

\bibitem[Stefanski(1990)]{Stefanski}
Leonard~A. Stefanski.
\newblock {Rates of convergence of some estimators in a class of deconvolution
  problems}.
\newblock \emph{{Statistics and Probability Letters}}, 9:\penalty0 229--235,
  1990.

\bibitem[Stefanski and Carroll(1990)]{Stefanski_Carroll}
Leonard~A. Stefanski and Raymond~J. Carroll.
\newblock {Deconvoluting kernel density estimators}.
\newblock \emph{Statistics}, 21:\penalty0 129--184, 1990.

\end{thebibliography}
\end{document}